\documentclass{amsproc}
\usepackage{amssymb}
\usepackage{amsfonts}

\theoremstyle{plain}

\newtheorem{corollary}{Corollary}

\newtheorem{definition}{Definition}

\newtheorem{proposition}{Proposition}

\newtheorem{theorem}{Theorem}
\numberwithin{equation}{section}
\input{tcilatex}

\begin{document}
\title[Property $\widetilde{\Omega }$]{Stein manifolds $M$ for which $%
O\left( M\right) $ has the property $\widetilde{\Omega }$}
\author{Aydin Aytuna }
\address{MDBF Sabanci University, Orhanli, 34956 Istanbul, Turkey}
\email{aytuna@sabanciuniv.edu}
\date{May 23, 2013}
\dedicatory{Dedicated to Mikhail Mikhaylovich Dragilev on the occation of
his 90 th birthday}

\begin{abstract}
In this note, we consider the linear topological invariant $\widetilde{%
\Omega }$\ \ for Fr\'{e}chet spaces of global analytic functions on Stein
manifolds. We show that $O\left( M\right) ,$ for a Stein manifold $M,$
enjoys the property $\widetilde{\Omega }$ \ if and only if every compact
subset of $M$ lies in a relatively compact sublevel set of a bounded
plurisubharmonic function defined on $M.$ We also look at some immediate
implications of this characterization.
\end{abstract}

\maketitle

\section{Introduction}

Spaces of analytic functions, regarded as an important class of nuclear Fr%
\'{e}chet spaces contributed amply to the development of the structure
theory of Fr\'{e}chet spaces. A profound example is the pioneering result of
Dragilev \cite{DRA1} on the absoluteness of bases in the space of analytic
functions on the unit disc with the usual topology. This paved the way to
the far-reaching theorem of Dynin-Mitiagin \cite{DM} on the absoluteness of
bases in \textit{every} nuclear Fr\'{e}chet space. Many more examples could
readily be provided. Of course this influence has not been one-sided.
Techniques and concepts from functional analysis were extensively used in
complex analysis. Advances in the structure theory of Fr\'{e}chet spaces,
found some applications in the Mitiagin-Henkin \cite{MH} program on the
linearization of basic results of the theory of analytic functions. (See,
for example \cite{AY1},\cite{ZA1},\cite{AY2}). In order to use the results
of the structure theory of Fr\'{e}chet spaces effectively it is imperative
to analyze the complex analytic properties shared by the complex manifolds
whose analytic function spaces possess a common linear topological
invariant. The present note is written from this perspective and aims to
characterize Stein manifolds whose analytic function spaces possess the
property $\widetilde{\Omega }$ of Vogt \cite{V1}. (See section 1 for the
definition)

Throughout this note we will denote the space of analytic functions on a
Stein manifold $M$ with the compact-open topology by $O\left( M\right) $.

In the first section we compile some background material for the linear
topological invariant $\widetilde{\Omega }$.

The second section is devoted to the proof of the characterization of Stein
manifolds $M$ for which $O\left( M\right) $ has the property $\widetilde{%
\Omega }$ as those manifolds with the property that every compact set of $M$
lie in a precompact sublevel set of a suitably chosen bounded
plurisubharmonic function (on $M)$.

Considering the class of Stein manifolds $M$ for which $O\left( M\right) $
has the property $\widetilde{\Omega },$ in the third section we show, as an
immediate corollary of the characterization theorem, the existence of
pluricomplex Green functions with certain special properties for this class.
The note ends with two examples among bounded domains in $\mathbb{C},$ one
in the class and one not in the class.

The manifolds considered in this note are always assumed to be connected. We
will use the standard terminology and results from functional analysis and
complex potential theory, as presented in \cite{V2} and \cite{K}
respectively. Throughout this note, the notation $\subset \subset $ will be
used to denote relatively compact containments.

\section{The linear topological invariant $\widetilde{\Omega }$}

In this section we give some background material on the linear topological
invariant $\widetilde{\Omega }.$

\begin{definition}
(Vogt \cite{V1}) Let $E$ be a Fr\'{e}chet space with a generating system of
seminorms $\left( \left\Vert .{}\right\Vert _{k}\right) _{k}$. $E$ is said
to have the property $\widetilde{\Omega },$ in case :

\begin{equation*}
\forall p\text{ }\exists \text{ }q,\text{ }d>0,\text{ }\forall k\text{ }%
\exists \text{ }C>0\text{ }\forall \varphi \epsilon E^{\ast }:\left\Vert
\varphi \right\Vert _{q}^{\ast }\text{ }\leq C\left( \left\Vert \varphi
\right\Vert _{p}^{\ast }\right) ^{\frac{d}{1+d}}\left( \left\Vert \varphi
\right\Vert _{k}^{\ast }\right) ^{\frac{1}{1+d}}
\end{equation*}
where $\left( \left\Vert .{}\right\Vert _{k}^{\ast }\right) _{k}$ are the
seminorms dual to $\left( \left\Vert .{}\right\Vert _{k}\right) _{k}.$
\end{definition}

\bigskip

Note that this property does not depend on the generating semi-norm system.
If $E$ is a nuclear Fr\'{e}chet space, it turns out that the conditions
below are also equivalent to the condition given in the definition of the
property $\widetilde{\Omega }:$

\bigskip

\begin{itemize}
\item There exists a closed bounded absolutely convex set $B$ in $E$ :%
\begin{equation*}
\forall p\text{ }\exists \text{ }q,\text{ }d>0,\text{ }\exists \text{ }C>0%
\text{ }\forall \varphi \epsilon E^{\ast }:\left\Vert \varphi \right\Vert
_{q}^{\ast }\text{ }\leq C\left( \left\Vert \varphi \right\Vert _{p}^{\ast
}\right) ^{\frac{d}{1+d}}\left( \left\Vert \varphi \right\Vert _{B}^{\ast
}\right) ^{\frac{1}{1+d}}
\end{equation*}

\item There exists a closed bounded absolutely convex set $B$ in $E$ :%
\begin{eqnarray*}
\forall p\text{ }\exists \text{ }q,\text{ }d &>&0,\text{ }C>0,\text{ such
that for all }r>0: \\
U_{q} &\subseteq &CrB+\frac{1}{r^{d}}U_{p}
\end{eqnarray*}

\item $\forall p$ $\exists $ $q,$ $d>0,$ $\forall k$ $\exists C>0,$ such
that for all $r>0:$%
\begin{equation*}
U_{q}\subseteq CrU_k+\frac{1}{r^{d}}U_{p}
\end{equation*}
where $U_{s}$ denotes the unit ball of the seminorm $\left\Vert
{}\right\Vert _{s}$, $s=1,2,...$ (see \cite{VW}, \cite{DMV}).
\end{itemize}

\bigskip

This property is stronger than $\Omega ,$ and is weaker than $\overline{%
\text{ }\Omega }$ conditions of Vogt, and as with all $\Omega $- type
invariants, is inherited by quotients \cite{V1}. This invariant plays an
important role in investigations of finding "non-polar" bounded sets in
nuclear Fr\'{e}chet spaces initiated by a question of P.Lelong. We refer
reader to \cite{DMV} for details on this matter.

Another interesting feature of nuclear Fr\'{e}chet spaces with the property $%
\widetilde{\Omega }$ is that continuous linear operators from such a space
into a nuclear weakly stable infinite type power series space are
necessarily compact \cite{V1}. In particular nuclear weakly stable infinite
type power series spaces, e.g. $O\left( M\right) ,$ for parabolic Stein
manifolds $M$ \cite{AS}, cannot not have the property $\widetilde{\Omega }. $
More generally we have,

\begin{proposition}
Let $X$ be a nuclear Fr\'{e}chet space. If the diametral dimension is equal
to the diametral dimension of an nuclear weakly stable infinite type power
series space the $X$ cannot have the property $\widetilde{\Omega }.$
\end{proposition}

\begin{proof}
Suppose that $X$ has the property $\widetilde{\Omega }$ and assume that the
diametral dimension of $X,$ $\Delta \left( X\right) ,$ satisfies $\Delta
\left( X\right) =\Delta \left( \Lambda _{\infty }\left( \alpha _{n}\right)
\right) $ for some nuclear weakly stable exponent sequence $\left( \alpha
_{n}\right) .$ Choose a generating seminorm system $\left( \left\Vert
.{}\right\Vert _{k}\right) _{k}$ so that $p+1$ is the index $q$ assigned to $%
p $ by the condition $\widetilde{\Omega }$.

Let $F\circeq \left\{ \left( x_{n}\right) :\sup \left\vert x_{n}\right\vert 
\text{ }d_{n}\left( U_{p+1},U_{p}\right) <\infty ,\forall p\right\} $ with
the natural Fr\'{e}chet space structure.

Since $F$ is in $\Delta \left( X\right) ,$ and $\Delta \left( \Lambda
_{\infty }\left( \alpha _{n}\right) \right) =\left( \left( x_{n}\right)
:\exists \text{ }R\geq 1,\text{ }\sup \left\vert x_{n}\right\vert R^{\alpha
_{n}}<\infty \right) ;$ in view of Grothendieck factorization theorem there
is an $R_{0}$ such that 
\begin{equation*}
\forall p\text{ , \ \ }\lim \sup_{n}\frac{-\ln d_{n}\left(
U_{p+1},U_{p}\right) }{\alpha _{n}}\leq \ln R_{0}\text{. }
\end{equation*}

On the other hand considering the usual topology on $\Delta \left( X\right)$ 
\cite{TT1}, which represents it as a projective limit of inductive limit of
Banach spaces, the continuous inclusion $\Delta \left( X\right) \subseteq
\Delta \left( \Lambda _{\infty }\left( \alpha _{n}\right) \right) $ gives:%
\begin{equation}
\forall R\geq 1\text{ and }p\text{ }\exists \text{ }q\text{, }%
C>0:\sup_{n}R^{\alpha _{n}}d_{n}\left( U_{q},U_{p}\right) \leq C.  \notag
\end{equation}

In particular : 
\begin{equation*}
\forall R\geq 1\text{ and }p\text{ }\exists \text{ }q\text{ :}\ln R\leq \lim
\inf_{n}\frac{-\ln d_{n}\left( U_{q},U_{p}\right) }{\alpha _{n}}
\end{equation*}

We now utilize the condition $\widetilde{\Omega }$, which in our notation,
reads as: There exists a closed bounded set $B\subseteq X$ such that:

\begin{eqnarray*}
\forall p\text{ }\exists \text{ }d &>&0\text{ },\text{ }C>0,\text{ such that
for all }r>0: \\
U_{p+1} &\subseteq &CrB+\frac{1}{r^{d}}U_{p}.
\end{eqnarray*}

Following the argument given in \cite{TT1}, we arrive at the estimate:

\begin{equation*}
\forall p\text{ }\exists \text{ }d>0\text{ },\text{ }C>0\text{, }-\ln
d_{n}\left( B,U_{p}\right) \leq \left( 1+d\right) \left( -\ln d_{n}\left(
U_{p+1},U_{p}\right) \right) +C\text{, \ \ }n=1,2,...\text{.}
\end{equation*}

Lets fix a $p$ and choose an $R>>\left( R_{0}\right) ^{\left( 1+d\right) }$
where $d$ is the constant appearing in the above equation$.$ Putting all the
above implications together, we get; 
\begin{align*}
\ln R&\leq \lim \inf_{n}\frac{-\ln d_{n}\left( U_{q},U_{p}\right) }{\alpha
_{n}}\leq \lim \inf_{n}\frac{-\ln d_{n}\left( B,U_{p}\right) }{\alpha _{n}}
\\
&\leq \lim \inf_{n}\frac{\left( 1+d\right) \left( -\ln d_{n}\left(
U_{p+1},U_{p}\right) \right) }{\alpha _{n}}\leq \left( 1+d\right) \ln R_{0}.
\end{align*}

This contradiction finishes the proof of the proposition.
\end{proof}

\bigskip

We would like to finish this section by making some immediate observations,
in view of the things said above, about the class of Stein manifolds whose
analytic function spaces have the property $\widetilde{\Omega }.$ Smoothly
bounded domains of holomorphy in $\mathbb{C}^{n},$ complete bounded Reinhard
domains, more generally hyperconvex Stein manifolds belong to this class
since their analytic function spaces possess a stronger property $\overline{%
\Omega }$ \cite{ZA2} \cite{AY3}. On the other hand $\mathbb{C}^{d}$, $%
d=1,2,..$, or more generally, parabolic Stein manifolds do not belong to
this class \cite{AS}.

\bigskip

\section{Main result}

In this section we give a characterization of Stein manifolds $M$ for which $%
O\left( M\right) $ has the property $\widetilde{\Omega }.$

\begin{theorem}
Let $M$ be a Stein manifold. The Fr\'{e}chet space $O\left( M\right) $ has
the property $\widetilde{\Omega }$ if and only if for every compact subset $K
$ of $M$ there exists a negative plurisubharmonic function $\varphi $ on $M$
and a $\alpha <0$ such that%
\begin{equation*}
K\subset \left( z\epsilon M:\varphi \left( z\right) <\alpha \right) \subset
\subset M.
\end{equation*}
\end{theorem}

\begin{proof}
Throughout the proof we will use the notation of Lemma 1 of \cite{AY3}. To
this end we fix a hermitian metric on $M$, and denote by $d\varepsilon $ the
measure $cd\mu $ where $\mu $ is the measure (equivalent to the volume form)
and $c$ is the positive continuous function, respectively, of Lemma 1 \cite%
{AY3}. We also choose a $C^{\infty }$ strictly plurisubharmonic exhaustion
function $\sigma $ of $M$ and let, 
\begin{equation*}
D_{n}\circeq \left( z\epsilon M:\sigma \left( z\right) <n\right) \text{, }%
n=1.2....\text{.}
\end{equation*}

$\left( \Rightarrow \right) $ It suffices to show that each $K_{n}\circeq 
\overline{D_{n}}$, $n=1,2,...$, is contained in a relatively compact
sub-level set of a bounded plurisubharmonic function. To this end fix a $%
K_{n_{0}}.$ Choose, as in \cite{ZA3} (\cite{ZA1}), a Hilbert space $\left(
H_{0},{\lbrack .{}\rbrack_{0}}\right) $ with continuous injections, 
\begin{equation*}
O\left( K_{n_{0}}\right) \hookrightarrow H_{0}\hookrightarrow AC\left(
K_{n_{0}}\right)
\end{equation*}
where $O\left( K_{n_{0}}\right) $ denotes the germs of analytic functions on 
$K_{n_{0}}$ with the usual inductive limit topology and $AC\left(
K_{n_{0}}\right) $ denotes the closure, in $C\left( K_{n_{0}}\right) ,$ of
the restriction of $O\left( K_{n_{0}}\right) $ to $K_{n_{0}}.$ For $n>n_{0}$%
, the pair $\left\{ K_{n_{0}},D_{n}\right\} $ is a regular pair in the sense
of \cite{ZA3} and hence the relative extremal function 
\begin{equation*}
\omega _{n}\left( z\right) \circeq \sup \left\{ u\left( z\right) :u\epsilon
PSH\left( D_{n}\right) ,\text{ }u\leq -1\text{ on }K_{n_{0}}\text{ and }%
u\leq 0\text{ on }D_{n}\right\}
\end{equation*}
is a continuous function on $D_{n}$ \cite{ZA3}. Clearly $\left( \omega
_{n}\right) _{n>n_{0}}$ forms a decreasing sequence of plurisubharmonic
functions. For $k=1,2,...$we define a norm on $O\left( M\right) $ by:%
\begin{equation*}
\left[ f\right] _{k}\circeq \left( \int_{D_{k+n_{0}}}\left\vert f\right\vert
^{2}d\varepsilon \right) ^{\frac{1}{2}},\text{ \ \ }f\epsilon O\left(
M\right) .
\end{equation*}

We will denote the corresponding Hilbert spaces by $H_{k}$, $k=1,2,...$. The
norm system $\left( \left[ \ast \right] _{k}\right) _{k=0}^{\infty }$
generates the topology of $O\left( M\right) $. Denoting the dual norms by $%
\left( \left[ \ast \right] _{k}^{\ast }\right) _{k=0}^{\infty }$, there
exists, in view of our assumption, an index $n_{1}$ and $d>0$ so that,

\begin{equation*}
\forall k\text{ }\exists \text{ }C>0\text{ }:\text{ }\left[ .{}\right]
_{n_{1}}^{\ast }\leq C\left( \left[. {}\right] _{k}^{\ast }\right) ^{\frac{d%
}{1+d}}\left( \left[. {}\right] _{0}^{\ast }\right) ^{\frac{1}{1+d}}.
\end{equation*}

Fix an $m>n_{1}.$ The inclusion $\iota _{m}:H_{m}\hookrightarrow H_{0}$,
being a compact continuous operator, can be represented as 
\begin{equation*}
\iota _{m}\left( x\right) =\sum\limits_{n}\lambda _{n}\langle x,f_{n}\rangle
_{m}e_{n}\text{, \ \ \ }\forall n,\text{ \ }\lambda _{n}\geq 0\text{ \ },%
\text{ }\lim \lambda _{n}=0,
\end{equation*}
for some orthonormal sequences $\left( f_{n}\right) _{n}$, $\left(
e_{n}\right) _{n}$ of $H_{m}$ and $H_{0}$ respectively. Let 
\begin{equation*}
d_{n}\circeq -\ln \lambda _{n},\text{ }n=1,2,...
\end{equation*}

We will regard, $\iota _{m}$ as inclusion and identify $f_{n}$ with $\lambda
_{n}e_{n},$ $\ \ n=1,2,...$. It is shown in \cite{AY3} (\cite{ZA3}) that $%
\left( e_{n}\right) _{n}$ forms a basis of $O\left( D_{m+n_{0}}\right) $ and
that this space can be represented as a finite center of the Hilbert scale
generated by $H_{m}$and $H_{0}.$ Moreover the coordinate functionals $\left(
e_{n}^{\ast }\right) _{n}$ on $O\left( M\right) $ satisfy 
\begin{equation*}
\left[ e_{n}^{\ast }\right] _{m}^{\ast }=e^{-d_{n}},\ n=1,2,\ .
\end{equation*}

In view of Proposition I.11 of \cite{AY3} (\cite{ZA3}), the relative
extremal function can be represented as: 
\begin{equation*}
1+\omega _{n_{0}+m}\left( z\right) =\lim \sup_{\xi \rightarrow z}\lim
\sup_{n}\frac{\ln \left\vert e_{n}\left( \xi \right) \right\vert }{d_{n}}\ \
\ \ \forall z\epsilon D_{n_{0}+m}\setminus K_{n_{0}}.
\end{equation*}

Fix an $\beta ,$ with $0$ $<\beta <\frac{d}{1+d}.$ In view of Hartogs lemma $%
\left( \text{Theorem 2.6.4 \cite{K}}\right) $:

\begin{equation*}
\forall \epsilon >0\text{ \ }\exists \text{ }C>0\text{ }:\text{ \ \ }%
\left\vert e_{n}\right\vert _{K_{\beta }}\leq Ce^{\beta d_{n}}
\end{equation*}
where $\left\vert {.}\right\vert _{_{\Gamma _{\beta }}}$ denotes the sup
norm on the precompact sub-level set 
\begin{equation*}
\Gamma _{\beta }\circeq \left( z\epsilon D_{n_{0}+m\text{ \ \ }}:1+\omega
_{n_{0}+m}\left( z\right) \leq \beta \right) .
\end{equation*}

For a given $f\epsilon O\left( M\right) $ we estimate on $\Gamma _{\beta }$ $%
:$%
\begin{align*}
\left\vert f\left( z\right) \right\vert & \leq \sum\limits_{n}\left\vert
e_{n}^{\ast }\left( f\right) \right\vert \left\vert e_{n}\left( z\right)
\right\vert \leq C\sum\limits_{n}\left[ e_{n}^{\ast }\right] _{n_{1}}^{\ast }%
\left[ f\right] _{n_{1}}e^{\beta d_{n}} \\
& \leq C\sum\limits_{n}\left( \left[ e_{n}^{\ast }\right] _{m}^{\ast
}\right) ^{\frac{d}{1+d}}\left( \left[ e_{n}^{\ast }\right] _{0}^{\ast
}\right) ^{\frac{1}{1+d}}\left[ f\right] _{n_{1}}e^{\beta d_{n}} \\
& \leq \widetilde{C}\left( \sum\limits_{n}e^{\left( \beta -\text{ }\frac{d}{%
1+d}\right) d_{n}}\right) \left[ f\right] _{n_{1}}\leq \widehat{C}\left[ f%
\right] _{n_{1}}
\end{align*}%
since $\left( d_{n}\right) =O\left( n^{\frac{1}{\dim M}}\right) $ (\cite{ZA1}%
). Moreover from the definition of $\left[ .{}\right] _{n_{1}},$ there is a
constant $C$ which does not depend upon $f$ such that%
\begin{equation*}
\left[ f\right] _{n_{1}}\leq C\left\vert f\right\vert _{K_{n_{0}+n_{1}+1}}
\end{equation*}%
where $\left\vert .{}\right\vert _{K_{n_{0}+n_{1}+1}}$ denotes the sup norm
on $K_{n_{0}+n_{1}+1}.$ Hence we have the estimate 
\begin{equation*}
\exists \text{ }C_{1}>0:\left\vert f\right\vert _{\Gamma _{\beta }}\leq
C_{1}\left\vert f\right\vert _{K_{n_{0}+n_{1}+1}},\forall f\epsilon O\left(
M\right) ,
\end{equation*}%
between the sup norms. By considering powers, as usual, we can take $C_{1}=1,
$ and also taking into account that $K_{n_{0}+n_{1}+1}=\overline{%
D_{n_{0}+n_{1}+1}}$ is holomorphically convex in $M,$ we see that 
\begin{equation*}
K_{n_{0}}\subseteq \left( z\epsilon D_{n_{0}+m\text{ \ \ }}:1+\omega
_{n_{0}+m}\left( z\right) \leq \beta \right) \subseteq \overline{%
D_{n_{0}+n_{1}+1}}\subset \subset M
\end{equation*}%
for a fixed $\beta ,$ with $0$ $<\beta <\frac{d}{1+d}$ and for every $m$ $>$%
\ $n_{1}.$

We let 
\begin{equation*}
\omega ^{n_{0}}\circeq \lim_{m}\text{ }\omega _{n_{0}+m}.
\end{equation*}

Being the limit of a decreasing sequence of plurisubharmonic functions, $%
\omega ^{n_{0}}$ is a negative plurisubharmonic function on $M$ and is
identically $-1$ on $K_{n_{0}}.$ Moreover for any $\beta $ with $\ 0$ $%
<\beta <\frac{d}{1+d},$ we have: 
\begin{equation*}
K_{n_{0}}\subseteq \left\{ z\epsilon M:\omega ^{n_{0}}<\beta \right\}
\subseteq \overline{D_{n_{0}+n_{1}+1}}\subset \subset M.
\end{equation*}

$\left( \Leftarrow \right) $ In this part of the proof we will follow the
argument given in Th1 of \cite{AY3} rather closely. Using the notation fixed
at the beginning of the proof we fix a generating system for $O\left(
M\right) $ given by the norms, 
\begin{equation*}
\left\Vert f\right\Vert _{k}\circeq \left( \int_{D_{k}}\left\vert
f\right\vert ^{2}d\varepsilon \right) ^{\frac{1}{2}}
\end{equation*}

where $D_{k}$ =$\left( z\epsilon M:\sigma \left( z\right) \leq k\right) ,$ $%
k=1,2,..\ .$As usual we will use the notation $U_{k}$ to denote the unit
ball corresponding to $\left\Vert {}\right\Vert _{k}$, $k=1,2,...$.

Let $k_{0}\epsilon \mathbb{N}$ be given. By our assumption there is a
negative plurisubharmonic function $\rho $ on $M$ and $\alpha _{1}<0,$ such
that,%
\begin{equation*}
\overline{D_{k_{0}}}\subseteq \left( z\epsilon M:\rho \left( z\right)
<\alpha _{1}\right) \subset \subset M.
\end{equation*}

Choose negative numbers $\alpha _{0}<$ $\alpha _{1}$, $\alpha _{2}$ and $%
k_{1}\epsilon \mathbb{N},$ $k_{0}<<$ $k_{1}$ such that 
\begin{align*}
\overline{D_{k_{0}}}&\subseteq \left( z\epsilon M:\rho \left( z\right)
<\alpha _{0}\right) \subseteq \left( z\epsilon M:\rho \left( z\right)
<\alpha _{1}\right) \\
&\subset \subset D_{k_{1}}\subseteq \left( z\epsilon M:\rho \left( z\right)
<\alpha _{2}\right) .
\end{align*}
and let 
\begin{equation*}
\Omega _{-}\circeq D_{k_{1}},\text{ \ \ \ \ }\Omega _{+}\circeq \overline{%
\left( z\epsilon M:\rho \left( z\right) <\alpha _{1}\right) }^{c}.
\end{equation*}

For a fixed $t>0,$ we let, 
\begin{equation*}
\rho _{t}\left( z\right) \circeq -\frac{t}{\alpha _{0}}\rho \left( z\right)
+t.
\end{equation*}
Clearly $\rho _{t}$ is a bounded plurisubharmonic function on $M$. \ 

Fix an $f\epsilon O\left( M\right) $ with $\left\Vert f\right\Vert
_{k_{1}}\circeq \left( \int_{D_{k_{1}}}\left\vert f\right\vert
^{2}d\varepsilon \right) ^{\frac{1}{2}}\leq 1.$

For such an $f,$ we have the estimate,%
\begin{equation*}
\int_{\Omega _{-}\cap \Omega _{+}}\left\vert f\right\vert ^{2}e^{-\rho
_{t}}d\mu \leq C\sup_{w\epsilon \Omega _{-}\cap \Omega _{+}}e^{-\rho
_{t}\left( w\right) }\leq Ce^{-\lambda t}
\end{equation*}
for some $C>0$ where $\lambda \circeq 1-\frac{\alpha _{1}}{\alpha _{0}}.$

In view of Lemma 1 of \cite{AY3} we can decompose $f$ on $\Omega _{-}\cap
\Omega _{+}$ as $f=f_{+}+f_{-}$ with $\ f_{+}\epsilon O\left( \Omega
_{+}\right) ,$ $f_{-}\epsilon O\left( \Omega _{-}\right) ;$ moreover, 
\begin{equation*}
\int_{\Omega _{+}}\left\vert f_{+}\right\vert ^{2}e^{-\rho _{t}}d\varepsilon
\leq Ce^{-\lambda t},\text{ \ \ \ \ }\int_{\Omega _{-}}\left\vert
f_{-}\right\vert ^{2}e^{-\rho _{t}}d\varepsilon \leq Ce^{-\lambda t}
\end{equation*}
for some constant $C>0$ which is independent of $f$ and $t.$

Hence,%
\begin{equation*}
\int_{\Omega _{+}}\left\vert f_{+}\right\vert ^{2}d\varepsilon \leq
Ce^{t\left( 1-\lambda \right) },\text{ \ \ \ }\int_{\Omega _{-}}\left\vert
f_{-}\right\vert ^{2}d\varepsilon \leq Ce^{t\left( 1-\lambda \right) }
\end{equation*}

Taking into account that $\rho _{t}\leq 0$ on $D_{k_{0}}$, we also have;

\begin{equation*}
\int_{D_{k_{0}}}\left\vert f_{-}\right\vert ^{2}d\varepsilon \leq
\int_{D_{k_{0}}}\left\vert f_{-}\right\vert ^{2}e^{-\rho _{t}}d\varepsilon
\leq Ce^{-\lambda t}.
\end{equation*}

Set 
\begin{equation*}
F = \left\{ 
\begin{array}{cc}
f_{+} & \text{ on } \Omega _{+} \\ 
f-f_{-} & \text{ on } \Omega _{-}%
\end{array}
\right.
\end{equation*}

The function $F$ is analytic on $M$ and from above we see that there is a $%
K>0:$ 
\begin{equation*}
\int \left\vert F\right\vert ^{2}d\varepsilon \leq Ke^{t\left( 1-\lambda
\right) }.
\end{equation*}

Also from the above considerations we have $:$%
\begin{equation*}
\int_{D_{k_{0}}}\left\vert F-f\right\vert ^{2}d\varepsilon
=\int_{D_{k_{0}}}\left\vert f_{-}\right\vert ^{2}d\varepsilon \leq
Ce^{-\lambda t}
\end{equation*}

Now let 
\begin{equation*}
B\circeq \left( g\epsilon O\left( M\right) :\int \left\vert g\right\vert
^{2}d\varepsilon \leq 1\right) .
\end{equation*}

Setting $r=e^{t\left( 1-\lambda \right) },$ the analysis above can be
summarized as:%
\begin{equation*}
\forall \text{ }k_{0\text{ \ }}\exists \text{ }k_{1\text{ \ }}\text{and }C>0:%
\text{\ \ \ }U_{k_{1}}\subseteq \frac{1}{r^{\frac{\lambda }{1-\lambda }}}%
U_{k_{0}}+CrB\text{ \ \ }\forall r\geq 1.
\end{equation*}

Since the inclusion above is trivially true for $\ 0<r\leq 1$ we conclude
that $O\left( M\right) $ has the property $\widetilde{\Omega }.$

This finishes the proof of the theorem.
\end{proof}

\bigskip

\section{Concluding Remarks}

Although the assignment $M\rightarrow O\left( M\right) $ from Stein
manifolds, into Fr\'{e}chet spaces is not a complete invariant, often,some
complex potential theoretic properties of the given manifold $M$ can be
deduced from the knowledge of the type of the Fr\'{e}chet space $O\left(
M\right) .$ We will look for a case in point in the context of the property $%
\widetilde{\Omega }.$

Let $M$ be a Stein manifold and $z_{0}\epsilon M.$ Recall that the
pluricomplex Green function $g_{M}$ $\left( \ast ,z_{0}\right) $ of $M$ with
pole at $z_{0}$ is the plurisubharmonic function on $M\ $defined as:%
\begin{align*}
g_{M}\left( z,z_{0}\right) & =\lim \sup_{\xi \rightarrow z}\text{ }\{\sup
u\left( \xi \right) :\text{ }u\text{ }\epsilon PSH\left( M\right) \text{ }%
,u\leq 0,\text{ and } \\
& \text{(in the local coordinates) \ }u\left( w\right) -\log \left\Vert
w-z_{0}\right\Vert \leq O\left( 1\right) \text{ as }w\rightarrow z_{0}\}
\end{align*}%
(see \cite{K} and the references given there). In one variable it coincides
with the classical Green function and as is well known, they exists if and
only if the space is not parabolic. Moreover if it exists, it is harmonic
off its pole hence is a very "regular" function. The situation is rather
different in higher dimensions. (\cite{K}, p.232). For example, denoting the
unit disc by $\Delta $, if we look at the domain $\mathbb{C}\times \Delta $ $%
\subseteq \mathbb{C}^{2},$ we immediately see that $g_{\mathbb{C}\times
\Delta }\left( \left( z,w\right) ,\mathbf{0}\right) =\log \left\vert
w\right\vert ;$ so the pluricomplex Green function is identically $-\infty $
on the whole complex line $\mathbb{C\times }\left( 0\right) .$

Let us call a plurisubharmonic function $u:M\rightarrow \left[ -\infty
,\infty \right) $ \textit{semi-proper }in case there exists a number $c$
such that $\left( z\epsilon M:u\left( z\right) <c\right) $ is non-empty and
is relatively compact in $M.$ As a corollary of our theorem we have,

\begin{corollary}
Let $M$ be a Stein manifold and assume that $O\left( M\right) $ has the
property $\widetilde{\Omega }.$ Then for each $z_{0}\epsilon M,$ the
pluricomplex Green function
\end{corollary}

$g_{M}$ $\left( \ast ,z_{0}\right) $ is semi proper and satisfies $g_{M}$ $%
\left( \ast ,z_{0}\right) ^{-1}\left( -\infty \right) =\left( z_{0}\right) .$

\begin{proof}
Fix $z_{0}\epsilon M,$ and choose a compact set $K$ containing $z_{0}$ in
its interior$.$ In view of the theorem above, there exists a negative
plurisubharmonic function $\sigma $ on $M$ and $c>0$ such that 
\begin{equation*}
K\subseteq \left( z\epsilon M:\sigma \left( z\right) <-c\right) \subset
\subset M.
\end{equation*}

Let $-c^{+}\circeq \max_{z\epsilon K}\sigma \left( z\right) $ and set $%
\widehat{\sigma }\circeq \sigma +c^{+}.$ We choose a strictly pseudoconvex, $%
D\subset \subset M$ with%
\begin{equation*}
K\subseteq \left( z\epsilon M:\widehat{\sigma }\left( z\right) <0\right)
\subseteq \left( z\epsilon M:\widehat{\sigma }\left( z\right) <\alpha
\right) \subset \subset D\subset \subset M
\end{equation*}
where $\alpha \circeq c^{+}-c.$ We let $\rho \circeq $ $g_{D}$ $\left( \ast
,z_{0}\right) .$ The plurisubharmonic function $\rho $ is a nice function,
in the sense that $e^{\rho }$ is continuous on $\overline{D}$ (\cite[%
Corollary 6.2.3]{K}). We fix $r_{1}<r_{2}<0$ so that 
\begin{equation*}
\left( z\epsilon D:\rho <r_{1}\right) \subseteq K\subseteq \left( z\epsilon
M:\widehat{\sigma }\left( z\right) <\alpha \right) \subset \subset \left(
z\epsilon D:\rho <r_{2}\right) \subset \subset D\subset \subset M.
\end{equation*}

Finally set%
\begin{equation*}
\Phi \circeq \left( \frac{r_{2}-r_{1}}{\alpha }\right) \widehat{\sigma }%
+r_{1}.
\end{equation*}

We will consider the open sets%
\begin{equation*}
U\circeq \left( z\epsilon D:\rho <r_{2}\right) ,\text{ \ \ }V\circeq 
\overline{\left( z\epsilon D:\rho <r_{1}\right) }^{c}\cap \left( z\epsilon
D:\rho <r_{2}\right)
\end{equation*}
of $D.$ \ For any $z\epsilon \partial V\cap U$, $\lim \sup_{\xi \rightarrow
z}\Phi \left( \xi \right) \leq \rho \left( z\right) $, by construction.
Hence in view of Corollary 2.9.15 \cite{K}, the function $u$ defined by; 
\begin{equation*}
u\circeq \left\{ 
\begin{array}{cc}
\max \left( \rho ,\Phi \right) & \text{ on }V \\ 
\rho & \text{ on }U-V%
\end{array}%
\right.
\end{equation*}
is a plurisubharmonic function on $U\circeq \left( z\epsilon D:\rho
<r_{2}\right)$.

Moreover on $\overline{\left( z\epsilon M:\widehat{\sigma }\left( z\right)
<\alpha \right) }^{c}\cap \left( z\epsilon D:\rho <r_{2}\right) $, $\max
\left( \rho ,\Phi \right) =\Phi .$ Hence we can extend $u$ to a bounded
plurisubharmonic function on whole of $M$ by setting $u$ to be equal to $%
\Phi $ outside $\left( z\epsilon D:\rho <r_{2}\right) .$ Now $u-\sup_{M}u,$
is a semi-proper negative plurisubharmonic function and since near $z_{0},$
it is equal to $g_{D}$ $\left( \ast ,z_{0}\right) -\sup_{M}u,$ 
\begin{equation*}
g_{D}\left( \ast ,z_{0}\right) \geq u-\sup_{M}u\text{ }
\end{equation*}
on $M.$ From this, it follows that $g_{D}\left( \ast ,z_{0}\right) $ is a
semi-proper plurisubharmonic function with $g_{D}\left( \ast ,z_{0}\right)
^{-1}\left( -\infty \right) =\left( z_{0}\right) .$ This finishes the proof
of the corollary.
\end{proof}

\bigskip

We would like to finish this note by looking at two simple, yet typical
examples.

The first example we want to look at is the punctured unit disc, $\Delta
-\left\{ 0\right\} .$ Since every bounded plurisubharmonic function on it
extends to a bounded plurisubharmonic function on the unit disc, it is not
possible to put, say $K=\left( z\epsilon \mathbb{C}:\left\vert z\right\vert =%
\frac{1}{2}\right) ,$ into a precompact sublevel set of a bounded
plurisubharmonic function on $\Delta -\left\{ 0\right\} $ in view of the
maximum principle for plurisubharmonic functions. Actually it is not
difficult to see that $O\left( \Delta -\left\{ 0\right\} \right) $ is
isomorphic to $O\left( \Delta \right) \times O\left( \mathbb{C}\right) $ as
Fr\'{e}chet spaces. Hence $O\left( \Delta -\left\{ 0\right\} \right) $
admits $O\left( \mathbb{C}\right) $ as a quotient space and so can not have
the property $\widetilde{\Omega }.$

\bigskip

The second example we will look at is also a subdomain of the unit disc.
This time we will throw away infinite number of closed discs with radii
tending to zero along with the origin from the unit disc. To this end fix an 
$n_{0}$ such that the closed discs; 
\begin{equation*}
K_{n}\circeq \left( z\epsilon \mathbb{C}:\left\vert z-\frac{1}{e^{n}}%
\right\vert \leq \frac{1}{e^{\frac{1}{n^{3}}}}\right)
\end{equation*}
are disjoint for $n\geq n_{0}.$ Let\ 
\begin{equation*}
\Omega \circeq \Delta -\left( \bigcup_{n\geq n_{0}}K_{n}\cup \left\{
0\right\} \right) .
\end{equation*}

Fix a holomorphically convex smoothly bounded compact subset $K$ of $\Omega .
$ Choose a subdomain $\Theta $ of $\Omega $ obtained from $\Delta $ by
deleting only finite number of $K_{n}^{\prime }s$ defined above such that it
contains $K$ as a holomorphically convex $\left( \text{in }\Theta \right) ,$
compact subset. Since $\Theta $ is hyperconvex (\cite{K}, p 80), the
relative extremal function $\omega _{K}^{\Theta }$ of $K,$ $\left( \text{ in 
}\Theta \right) ,$ is a continuous function and $\left( z\epsilon \Theta
:\omega _{K}^{\Theta }\left( z\right) =-1\right) =K$ (\cite{ZA3}). For
constants $c$ near $-1,$ the corresponding sublevel sets of $\omega
_{K}^{\Theta }$ restricted to $\Omega ,$ are precompact in $\Omega $ and
certainly they contain $K.$ Since we can find an exhaustion of $\Omega $ by
such compact sets $K,$ the space $O\left( \Omega \right) $ has the property $%
\widetilde{\Omega }$, in view of the theorem above. However, $O\left( \Omega
\right) $ does not have the stronger property $\overline{\Omega }$. This
follows because the radii $\left( r_{n}\right) _{n}$ of the deleted discs
satisfy;%
\begin{equation*}
\sum \frac{n}{\ln \left( \frac{1}{r_{n}}\right) }<\infty ,
\end{equation*}%
and hence, by a result of Zaharyuta \cite{ZA4}, $O\left( \Omega \right)
\ncong O\left( \Delta \right) .$ In fact not much is known about the linear
topological properties of the Fr\'{e}chet space $O\left( \Omega \right) .$

\end{document}